\documentclass[
submission
]{dmtcs-episciences}

\usepackage[utf8]{inputenc}
\usepackage{subcaption, tikz}
\usepackage{amssymb}
\usepackage{epstopdf}
\usepackage{amssymb}
\usepackage{amsthm, amsmath, float}

\newtheorem{definition}{Definition}

\newtheorem{theorem}{Theorem}
\newtheorem{lemma}{Lemma}

\newtheorem{corollary}{Corollary}

\newcommand{\MM}{\mathcal{M}}

\tikzset{dot/.style={draw,shape=circle,fill=black,scale=.2}}

   \tikzset{*--*/.style={decoration={
   markings,
   mark=at position 0 with { \draw [fill] (0,0) circle [radius=1pt];},
   mark=at position 1 with { \draw [fill] (0,0) circle [radius=1pt];}},postaction={decorate}}}

\author{Megan A. Martinez \affiliationmark{1} \and Manda Riehl\affiliationmark{2}
 }
\title{A bijection between the set of nesting-similarity classes and L\&P matchings} 
\affiliation{
	Department of Mathematics, Ithaca College, Ithaca, NY\\
    Department of Mathematics, Rose-Hulman Institute of Technology}
\received{2017-4-28}

\revised{2017-9-8}

\accepted{2017-10-23}
\begin{document}
\publicationdetails{19}{2017}{2}{1}{3291}
\maketitle

\begin{abstract}
Matchings are frequently used to model RNA secondary structures; however, not all matchings can be realized as RNA motifs. One class of matchings, called the L \& P matchings, is the most restrictive model for RNA secondary structures in the Largest Hairpin Family (LHF). The L \& P matchings were enumerated in $2015$ by Jefferson, and they are equinumerous with the set of nesting-similarity classes of matchings, enumerated by Klazar. We provide a bijection between these two sets. This bijection preserves noncrossing matchings, and preserves the sequence obtained reading left to right of whether an edge begins or ends.
\end{abstract}


Ribonucleic acid (RNA) is an essential molecule found in the cells of all living things. Usually, RNA is formed by a string of nucleotides which folds over on itself, creating secondary bonds. The structure of these secondary bonds is a topic of great interest and has been studied from a biological and mathematical perspective \cite{Condon, Jefferson}. Mathematically, RNA secondary structures can be modeled by considering each nucleotide as a vertex and bonds between nucleotides that are not part of the RNA backbone as edges. Each vertex is incident to at most one edge and thus the graph obtained is a matching.

We represent these matchings as $2n$ points drawn along a horizontal line (the backbone) and arcs drawn between pairs of points represent the nucleotide bonds. For simplicity, we assume all matchings are \emph{complete} (i.e., all vertices are incident to an edge), and therefore contain $n$ edges. RNA often has nucleotides that are not bonded, but these structures can be reconstructed by adding any number of isolated vertices to a complete matching. We notate the set of complete matchings with $n$ edges by $\MM(n)$.

For a matching, $M$, we denote the set of edges $E(M)$ and label these edges with $[n]$ in increasing order from left to right by their left endpoints. For the edge labeled $i$, we write $i=(i_1,i_2)$ where $i$ represents the label of the edge and $i_1$, $i_2$ represent the position of the left and right endpoints of the edge, respectively.  A pair of edges $i=(i_1,i_2)$ and $j=(j_1, j_2)$ are said to be \emph{nested} if $i_1<j_1<j_2<i_2$ and \emph{crossing} if $i_1<j_1<i_2<j_2$. For a matching $M \in \MM(n)$, let $ne(M)$ and $cr(M)$ denote the number of pairs of nested edges and crossing edges in $M$, respectively. 
Additionally, the edges $i=(i_1,i_2)$ and $j=(j_1, j_2)$ are said to form a \emph{hairpin} if $j_1=i_1+1$ and $j_2=i_2+1$ (Figure~\ref{hairpin}). A nested sequence of edges that can be drawn above the backbone is called a \emph{ladder} (Figure~\ref{ladder4}).

\begin{figure}[H]
\centering
\begin{subfigure}[b]{0.15\textwidth}
\centering
\begin{tikzpicture}[scale=0.45] 
\foreach \x in {0,...,3}
	{
	\node[dot] at (\x,0)(\x){};
	\draw (\x,0)node[below]{\x};
	}

\draw(0)  to [bend left=90] (2);
\draw(1)  to [bend left=90] (3);

\end{tikzpicture}
\caption{} \label{hairpin}
\end{subfigure}
\begin{subfigure}[b]{0.35\textwidth}
\centering
\begin{tikzpicture}[scale=0.45] \label{ladder2}
\foreach \x in {0,...,7}
	{
	\node[dot] at (\x,0)(\x){};
	\draw (\x,0)node[below]{\x};
	}

\draw(0)  to [bend left=90] (7);
\draw(1)  to [bend left=90] (6);
\draw(2)  to [bend left=90] (5);
\draw(3)  to [bend left=90] (4);

\end{tikzpicture}
\caption{} \label{ladder4}
\end{subfigure}
\begin{subfigure}[b]{0.4\textwidth}
\centering
\begin{tikzpicture}[scale=0.45]
\foreach \i in {0,...,13}
	{
		\node[dot] at (\i,0)(\i){};
	}
	
\draw(0)  to [bend left=55]  node[midway,above]{1} (9);
\draw(1)  to [bend left=55] node[midway,above]{2}(6);
\draw(2)  to [bend left=55] node[midway,above]{3}(3);
\draw(4)  to [bend left=55] node[midway,above]{4}(13);
\draw(5)  to [bend left=55] node[midway,right=8pt, above]{5}(10);
\draw(7)  to [bend left=55] node[midway,above]{6}(8);
\draw(11)  to [bend left=55] node[midway,above]{7}(12);
\end{tikzpicture}
\caption{} \label{LPMatching}
\end{subfigure}
\caption{The left matching is a hairpin, the middle matching is a ladder of four edges, and the right matching is an example of an L \& P matching with edges labeled by left endpoint.} \label{ladder3}
\end{figure}
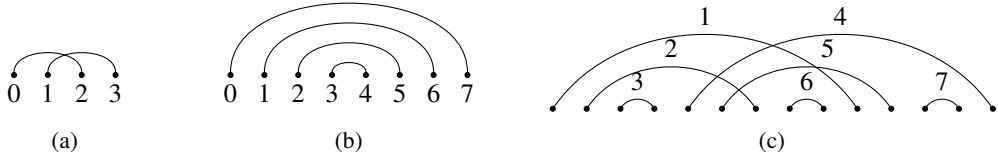

Matchings that contain no crossing pairs of edges are called \emph{noncrossing matchings}. The number of noncrossing matchings with $n$ edges is well-known to be counted by the $n$th Catalan number, $C_n=\frac{1}{n+1}\binom{2n}{n}$, and has been studied in several contexts, including pattern avoidance \cite{Bloom}.

For each matching, we can examine the vertices from left to right and list whether vertices are left or right endpoints of an edge.  For example, for the matching in Figure \ref{example of matching}, this sequence is LLRLLRRRLR; we call this the \emph{LR-sequence} of a matching. LR-sequences are in bijection with Dyck paths, and each noncrossing matching has a distinct LR-sequence. For any $M \in \MM(n)$, define $nc(M)$ to be the noncrossing matching with the same LR-sequence as $M$.

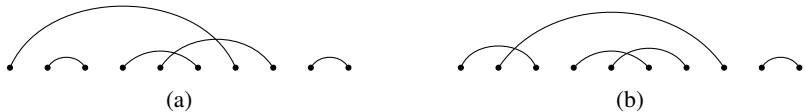
\begin{figure}[b]
\centering
	\begin{subfigure}[b]{0.4\textwidth}
	\centering
		\begin{tikzpicture}[scale=0.5]
		\foreach \i in {0,...,9}
			{
				\node[dot] at (\i,0)(\i){};
			}
	
		\draw(0)  to [bend left=65] (6);
		\draw(1)  to [bend left=55] (2);
		\draw(3)  to [bend left=45] (5);
		\draw(4)  to [bend left=55] (7);
		\draw(8)  to [bend left=55] (9);

		\end{tikzpicture}
		\caption{}
		\label{example of matching}
	\end{subfigure}
	\begin{subfigure}[b]{0.4\textwidth}
	\centering
	\begin{tikzpicture}[scale=0.5]
		\foreach \i in {0,...,9}
			{
				\node[dot] at (\i,0)(\i){};
			}
	
		\draw(0)  to [bend left=65] (2);
		\draw(1)  to [bend left=55] (7);
		\draw(3)  to [bend left=45] (5);
		\draw(4)  to [bend left=55] (6);
		\draw(8)  to [bend left=55] (9);

		\end{tikzpicture}
		\caption{} \label{nesting similarity b}
	\end{subfigure}
	\caption{Two nesting-similar matchings. Each matching has LR-sequence LLRLLRRRLR and two pairs of nested edges.}
	\label{nesting similarity}
\end{figure}

In order to model RNA structures that contain crossings, different families of matchings have been studied. In this paper, we focus on the family of L \& P matchings. The L \& P matchings were first rigorously defined by Condon et al.\ \cite{Condon} and this definition was later refined by Jefferson \cite{Jefferson}. Each matching in the L \& P family can be constructed inductively by starting from either a hairpin or a single edge, and either a) inflating an edge by a ladder, or b) inserting a noncrossing matching into an L \& P matching \cite{Jefferson}.  For an example of an L \& P matching, see Figure~\ref{LPMatching}.  The counting sequence for the number of L \& P matchings with $n$ edges begins 1, 3, 12, 51, 218, 926, 16323, 67866, 280746 and is given by the formula $2\cdot 4^{n-1}-\frac{3n-1}{2n+2}\binom{2n}{n}$ \cite{Jefferson}.

It was noted by Jefferson that L \& P matchings with $n$ edges, which we denote $\mathcal{LP}_n$, are equinumerous to the equivalence classes on matchings with $n$ edges determined by the nesting-similarity equivalence, $\sim_{ne}$, defined by Klazar \cite{Klazar}. We say that two matchings $M, N \in \mathcal{M}(n)$ are \emph{nesting-similar}, and write $M \sim_{ne} N$, if and only if $ne(M)=ne(N)$ and $M,N$ have the same LR-sequence. An example of this equivalence is shown in Figure \ref{nesting similarity}. Klazar showed that there are $2\cdot 4^{n-1}-\frac{3n-1}{2n+2}\binom{2n}{n}$ nesting-similarity equivalence classes for matchings in $\mathcal{M}(n)$; however, no explicit bijection was discovered connecting L \& P matchings to nesting-similarity classes. Our main result is to construct a bijective correspondence between these two structures.

Klazar \cite{Klazar} utilized a different, but equivalent, definition for nesting-similarity in his work. He examined the \emph{tree of matchings} where the vertex set is the infinite set of matchings, $\bigcup_{n=0}^{\infty} \MM(n)$. Matchings are connected if one can be obtained from the other by inserting a new edge whose left endpoint occurs earliest. With this construction, two matchings are said to be nesting-similar if we can record the number of nestings for the children of each matching, which have one added edge, and obtain the same multiset. It is straightforward to show that this definition is equivalent to the definition above.

In his paper, Klazar shows a correspondence between his nesting-similarity matchings and tunnel pairs in Dyck paths. Our bijection between L \& P matchings and the set of nesting-similarity classes will be a composition of two bijections. The composed mapping will pass through the set of nestings in noncrossing matchings, which correspond to tunnel pairs in Dyck paths. Define $\mathcal{NCN}_n = \{(M,a,b) \mid  M\in \mathcal{M}(n) \text{ noncrossing};a=b=0 \text{ or \allowbreak edges }  a<b\text{ are nested} \allowbreak \text{ in }M\}$; so, $\mathcal{NCN}_n$ is the set of noncrossing matchings with a chosen nested pair of edges ($a=b=0$ indicates no nested pair has been indicated). 

In Section~\ref{sec:LP}, we define a bijection between L \& P matchings and $\mathcal{NCN}_n$, and in Section~\ref{sec:nes}, we define a bijection between nesting-similarity classes and $\mathcal{NCN}_n$.Thus We show that both L \& P matchings and nesting-similarity classes are equinumerous to $\mathcal{NCN}_n$.

\section{L \& P matchings and noncrossing matchings} \label{sec:LP}

The process of constructing an L \& P matching implies that such a matching contains a crossing exactly if the matching can be built inductively from a hairpin.  As a result, any L \& P matching that contains a crossing will have all crossings occur in an inflated hairpin. Given a matching, $M$, we will label edges by left endpoint, as in Figure~\ref{LPMatching}. In this Figure, edges 1, 2 crossing edges 4, 5 comprise an inflated hairpin. Below we provide a precise definition of this structure.

\begin{definition}
A \emph{maximal inflated hairpin} in an L \& P matching is two sets of edges $A=\{a_1,a_2,\ldots,a_k\}$ and $B=\{b_1,b_2,\ldots,b_{\ell}\}$ such that
\begin{enumerate}
\item every pair of edges in $A$ and every pair of edges in $B$ is nested,
\item for every $a_i \in A$, $a_i$ crosses every edge in $B$, and
\item every crossing in $M$ occurs between edges in $A$ and $B$.
\end{enumerate}
We let $A$ be the set of edges with smaller labels (the left side of the inflated hairpin), and we say that $M$ contains the inflated hairpin $(A,B)$.
\end{definition}

Any L \& P matching consists of a maximal inflated hairpin with noncrossing matchings inserted below the hairpin as in Figure \ref{fig:structure}. It is possible for the inflated hairpin to be empty, yielding a noncrossing matching.
For example, the matching in Figure~\ref{example of matching} is L \& P: the first and third edges crossing the fourth edge form an inflated hairpin. However, the matching in Figure~\ref{nesting similarity b} is not L \& P, since the first four edges are all involved in crossings, but these four edges do not form an inflated hairpin.

Notice that every edge not in the inflated hairpin of an L \& P matching must begin and end between two vertices that are adjacent in the inflated hairpin, otherwise it would be part of an inflated hairpin itself. This fact will be used in our bijection.

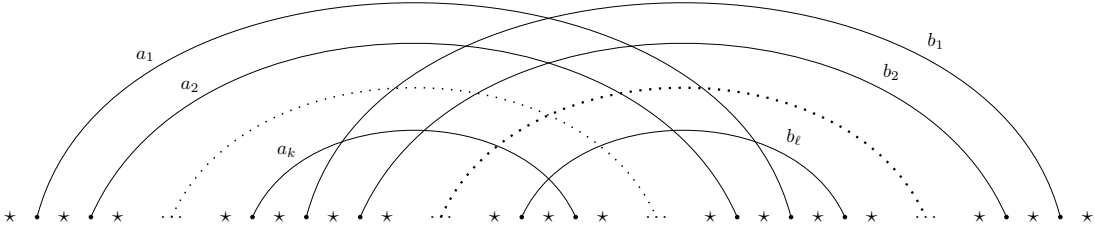
\begin{figure}[ht]
\resizebox{1\textwidth}{!}{
\begin{tikzpicture}[scale=1]
\node at (0,0){{\large$\star$}};
\node[dot] at (0.5,0)(a1){}; 
\node at (1,0){{\large$\star$}}; 
\node[dot] at (1.5,0)(a2){}; 
\node at (2,0){{\large$\star$}}; 
\node at (3,0)(1){$\ldots$};
\node at (4,0){{\large$\star$}};
\node[dot] at (4.5,0)(ak){}; 
\node at (5,0){{\large$\star$}};
\node[dot] at (5.5,0)(b1){}; 
\node at (6,0){{\large$\star$}}; 
\node[dot] at (6.5,0)(b2){}; 
\node at (7,0){{\large$\star$}}; 
\node at (8,0)(2){$\ldots$};
\node at (9,0){{\large$\star$}};
\node[dot] at (9.5,0)(bl){}; 
\node at (10,0){{\large$\star$}};
\node[dot] at (10.5,0)(ak2){}; 
\node at (11,0){{\large$\star$}};
\node at (12,0)(3){$\ldots$};
\node at (13,0){{\large$\star$}};
\node[dot] at (13.5,0)(a22){}; 
\node at (14,0){{\large$\star$}};
\node[dot] at (14.5,0)(a12){}; 
\node at (15,0){{\large$\star$}};
\node[dot] at (15.5,0)(bl2){}; 
\node at (16,0){{\large$\star$}};
\node at (17,0)(4){$\ldots$};
\node at (18,0){{\large$\star$}};
\node[dot] at (18.5,0)(b22){}; 
\node at (19,0){{\large$\star$}};
\node[dot] at (19.5,0)(b12){}; 
\node at (20,0){{\large$\star$}};

		\draw(a1)  to [bend left=75]node[near start,left=8pt]{$a_1$} (a12);
		\draw(a2)  to [bend left=66]node[near start,left=8pt]{$a_2$} (a22);
		\draw(ak)  to [bend left=65]node[near start,left=8pt]{$a_k$} (ak2);
		\draw(b1)  to [bend left=75]node[near end,right=8pt,above]{$b_1$} (b12);
		\draw(b2)  to [bend left=66]node[near end,right=8pt,above]{$b_2$} (b22);
		\draw(bl)  to [bend left=65]node[near end,right=8pt,above]{$b_{\ell}$} (bl2);

		\draw[loosely dotted,very thick](8,0)  to [bend left=66](17,0);
		\draw[loosely dotted, thick](3,0)  to [bend left=66](12,0);

\end{tikzpicture}
}
\caption{The structure of an L \& P matching; Each $\star$ indicates a position where a noncrossing matching may be inserted. The edges $a_1, a_2,\ldots, a_k$ cross the edges $b_1,b_2,\ldots,b_{\ell}$ to form the inflated hairpin.} \label{fig:structure}
\end{figure}

\begin{lemma} \label{lem:LP}
Let $M \in \mathcal{LP}_n$ such that $M$ contains an inflated hairpin $(A, B)$. Then the right endpoints of the inflated hairpin appear in the order $(a_k,a_k-1,\ldots,a_1,b_{\ell},b_{\ell}-1,\ldots,b_1)$ and can be rearranged to $(b_{\ell},b_{\ell}-1,\ldots,b_1,a_k,a_k-1,\ldots,a_1)$ to construct $nc(M)$. Additionally, every pair of edges in $A \cup B$ are nested in $nc(M)$.
\end{lemma}

\begin{proof}
Since we defined $A$ to contain edges with smaller labels than $B$, the result that the right endpoints are in the order $(a_k,a_k-1,\ldots,a_1,b_{\ell},b_{\ell}-1,\ldots,b_1)$ follows immediately. Now consider rearranging the right endpoints such that they appear in the order $(b_{\ell},b_{\ell}-1,\ldots,b_1,a_k,a_k-1,\ldots,a_1)$. It is easy to see that every pair of edges in $A \cup B$ is now nested and that no additional crossings can be created by this rearrangement.  Therefore the result is a noncrossing matching. Notice that, since noncrossing matchings are in direct bijection with LR-sequences, the resulting matching is in fact $nc(M)$.
\end{proof}

For an example of this conversion of an L \& P matching with an inflated hairpin into a noncrossing matching, see Figure \ref{fig:LPMatching}.  The results of Lemma \ref{lem:LP} will be useful as we define our bijection from $\mathcal{LP}_n$ to $\mathcal{NCN}_n$.

\begin{definition}
Define $\phi: \mathcal{LP}_n \rightarrow \mathcal{NCN}_n$ where,
\[\phi(M) = \begin{cases} (M,0,0), & \text{if $M$ noncrossing}, \\ (nc(M),\max(A),\max(B)), & \text{otherwise}. \end{cases}\]
where $(A,B)$ is the inflated hairpin of $M$.
\end{definition}

The results of Lemma \ref{lem:LP} imply that $\max(A), \max(B)$ are a nested pair of edges in $nc(M)$, and therefore $\phi$ is well-defined. All that remains is to show that $\phi$ is in fact a bijection.

\begin{theorem}\label{phi}
The mapping $\phi: \mathcal{LP}_n \rightarrow \mathcal{NCN}_n$ is a bijection.
\end{theorem}

\begin{proof}
We will construct the inverse of $\phi$. First note that $\phi^{-1}(M,0,0)=M$.

Now consider some noncrossing matching $M$ with nested pair of edges $(a,b)$. Let $A=\{a_i \in E(M) \mid (a_i,a)\text{ are nested and } a_i\leq a\text{ such that }a_1<a_2<\cdots<a_k=a\}$ and $B=\{b_i \in E(M) \mid (b_i,b)\text{ are nested and } a<b_i\leq b\text{ such that }b_1<b_2<\cdots<b_{\ell}=b\}$. Then, since edges are labeled by left endpoints and $M$ is noncrossing, the right endpoints of $A \cup B$ appear in the order $(b_{\ell},b_{\ell}-1,\ldots,b_1,a_k,a_k-1,\ldots,a_1)$.

Let $M'$ be the resulting matching when the right endpoints of $A \cup B$ are reordered to appear as $(a_k,a_k-1,\ldots,a_1,b_{\ell},b_{\ell}-1,\ldots,b_1)$. It is straightforward to show that $M'$ is an L \& P matching with inflated hairpin $(A,B)$ where $(\max(A),\max(B))=(a,b)$. It follows that $\phi^{-1}(M,a,b)=M'$. Therefore $\phi$ is a bijection, as desired.
\end{proof}

\begin{figure}[ht]
\centering
\begin{tikzpicture}[scale=0.5]
\foreach \i in {0,...,13}
	{
		\node[dot] at (\i,0)(\i){};
	}
	
\draw(0)  to [bend left=55]  node[midway,above]{1} (9);
\draw(1)  to [bend left=55] node[midway,above]{2}(6);
\draw(2)  to [bend left=55] node[midway,above]{3}(3);
\draw(4)  to [bend left=55] node[midway,above]{4}(13);
\draw(5)  to [bend left=55] node[midway,right=8pt, above]{5}(10);
\draw(7)  to [bend left=55] node[midway,above]{6}(8);
\draw(11)  to [bend left=55] node[midway,above]{7}(12);
\end{tikzpicture} \hspace{1cm}
\begin{tikzpicture}[scale=0.5]
\foreach \i in {0,...,13}
	{
		\node[dot] at (\i,0)(\i){};
	}
	
\draw(0)  to [bend left=55]  node[midway,above]{1} (13);
\draw(1)  to [bend left=55] node[midway,above]{2}(10);
\draw(2)  to [bend left=45] node[midway,above]{3}(3);
\draw(4)  to [bend left=60] node[midway,above]{4}(9);
\draw(5)  to [bend left=45] node[midway,above]{5}(6);
\draw(7)  to [bend left=45] node[midway,above]{6}(8);
\draw(11)  to [bend left=45] node[midway,above]{7}(12);
\end{tikzpicture}
\caption{On the left is an L \& P matching with labeled edges. In this matching, the inflated hairpin contains edges 1, 2 crossing edges 4, 5. On the right is the corresponding noncrossing matching obtained by swapping the right endpoints of edges in the inflated hairpin.}
\label{fig:LPMatching}
\end{figure}
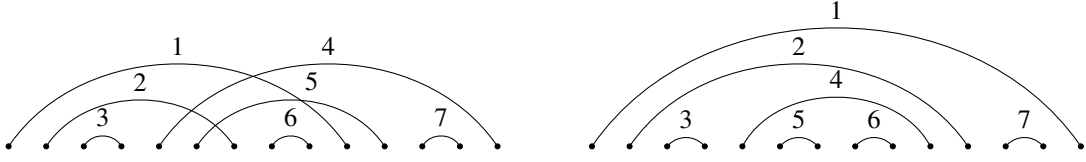

\section{Nesting-Similarity Classes and Noncrossing Matchings} \label{sec:nes}

Recall that $M \sim_{ne}N$ if and only if $ne(M)=ne(N)$ and $M,N$ have the same LR-sequence. Klazar showed that the nesting-similarity classes are equinumerous with tunnel pairs in Dyck Paths through the use of transpositions that swap the endpoints of nestings with minimal width. Using this map iteratively, for any LR-sequence, if $M$ is the corresponding noncrossing matching, then for every $i$ where $0\leq i \leq ne(M)$, Klazar proved that there exists some matching with the same LR-sequence and $i$ pairs of nested edges. Notice that the noncrossing matchings contain the maximum possible number of nestings for a particular LR-sequence, so Klazar's result encompasses all the nesting-similarity classes.

We define a bijection between nesting-similarity classes and the set $\mathcal{NCN}_n$ which explicitly defines a representative for each equivalence class. Let $nep(M)$ denote the list of nested pairs of edges in $M$ ordered lexicographically by prioritizing the second element. For example, $$nep(M)=\{(1,2), (1,3), (2,3), (1,4), (2,4), (1,5), (2,5), (4,5), (1,6), (2,6), (4,6), (1,7)\}$$ for the noncrossing matching of Figure~\ref{fig:LPMatching}. Let $rperm(M)$ be the order in which the right endpoints of edges appear. For example, $rperm(M)=3564271$. We define $rperm$ for all matchings, not necessarily noncrossing. However, when we have noncrossing matchings, rperm is useful for identifying nestings.

\begin{lemma} \label{lem:rperm}
In a noncrossing matching $M$, edges $a,b$ are nested, with $a<b$, if and only if $b$ appears before $a$ in $rperm(M)$.
\end{lemma}

This follows quickly from the method of labeling edges and the definition of nestings.

Consider a noncrossing matching $M$ with $k$ pairs of nested edges. Given some $i$ with $0 \leq i \leq k$, we define a process by which we rearrange the vertices in a noncrossing matching to obtain a matching with the same LR-sequence and $i$ pairs of nested edges. These matchings will be the representatives of the nesting-similarity classes.

In the definition below, given edges $a,b$ where $1 \leq a,b \leq n$ ($a \neq b$), let $(a,b).M$ denote the matching that results by swapping the left endpoints of edges $a$ and $b$.

\begin{definition}\label{def:swaps}
Let $M$ be a noncrossing matching with $k$ pairs of nested edges, and $nep(M)=\{(a_{1},b_{1}),\newline(a_{2},b_{2}),\ldots,(a_k,b_k)\}$. Define the sequence of matchings $M_0, M_1, M_2, \ldots, M_k$ by $M_0=M$ and $M_i=(a_i,b_i).M_{i-1}$ for all $i \in [k]$. Additionally, let $lperm(M_i)$ denote the order in which the left endpoints of matching $M_i$ appear.
\end{definition}

For an example of this definition, see Figure~\ref{fig:swaps}. It is clear that each $M_i$ has the same LR-sequence as $M$. Our goal is to additionally show that $ne(M_i)=k-i$. This would imply that the matchings in $\{M_0,M_1,\ldots,M_k \mid M_0 \text{ noncrossing}\}$ form a set of representatives of the $k+1$ distinct nesting-similarity classes for matchings with the same LR-sequence as $M$.

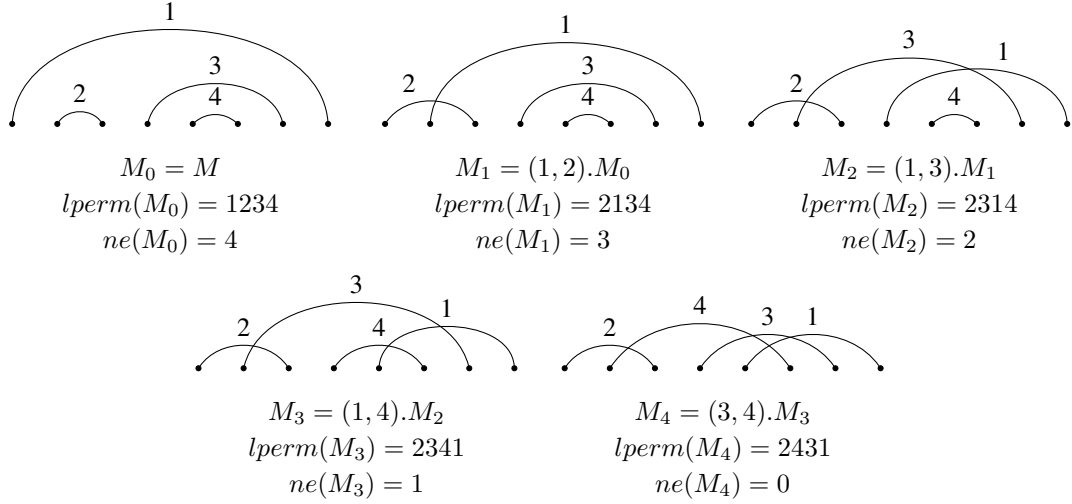
\begin{figure}[ht]
\centering
\begin{tikzpicture}[scale=0.6]
\foreach \i in {0,...,7}
	{
		\node[dot] at (\i,0)(\i){};
	}
	
\draw(0)  to [bend left=85]  node[midway,above]{1} (7);
\draw(1)  to [bend left=55] node[midway,above]{2}(2);
\draw(3)  to [bend left=75] node[midway,above]{3}(6);
\draw(4)  to [bend left=40] node[midway,above]{4}(5);
\draw(3.5,-1)node{$M_0=M$};
\draw(3.5,-1.8)node{$lperm(M_0)=1234$};
\draw(3.5,-2.6)node{$ne(M_0)=4$};
\end{tikzpicture} \hspace{5mm}
\begin{tikzpicture}[scale=0.6]
\foreach \i in {0,...,7}
	{
		\node[dot] at (\i,0)(\i){};
	}
	
\draw(1)  to [bend left=85]  node[midway,above]{1} (7);
\draw(0)  to [bend left=55] node[midway,left=8pt,above]{2}(2);
\draw(3)  to [bend left=75] node[midway,above]{3}(6);
\draw(4)  to [bend left=40] node[midway,above]{4}(5);
\draw(3.5,-1)node{$M_1=(1,2).M_0$};
\draw(3.5,-1.8)node{$lperm(M_1)=2134$};
\draw(3.5,-2.6)node{$ne(M_1)=3$};
\end{tikzpicture}\hspace{5mm}
\begin{tikzpicture}[scale=0.6]
\foreach \i in {0,...,7}
	{
		\node[dot] at (\i,0)(\i){};
	}
	
\draw(3)  to [bend left=85]  node[midway,,right=10pt,above]{1} (7);
\draw(0)  to [bend left=55] node[midway,above]{2}(2);
\draw(1)  to [bend left=75] node[midway,above]{3}(6);
\draw(4)  to [bend left=40] node[midway,above]{4}(5);
\draw(3.5,-1)node{$M_2=(1,3).M_1$};
\draw(3.5,-1.8)node{$lperm(M_2)=2314$};
\draw(3.5,-2.6)node{$ne(M_2)=2$};
\end{tikzpicture}\hspace{5mm}
\begin{tikzpicture}[scale=0.6]
\foreach \i in {0,...,7}
	{
		\node[dot] at (\i,0)(\i){};
	}
	
\draw(4)  to [bend left=85]  node[midway,above]{1} (7);
\draw(0)  to [bend left=55] node[midway,above]{2}(2);
\draw(1)  to [bend left=75] node[midway,above]{3}(6);
\draw(3)  to [bend left=55] node[midway,above]{4}(5);
\draw(3.5,-1)node{$M_3=(1,4).M_2$};
\draw(3.5,-1.8)node{$lperm(M_3)=2341$};
\draw(3.5,-2.6)node{$ne(M_3)=1$};
\end{tikzpicture}\hspace{5mm}
\begin{tikzpicture}[scale=0.6]
\foreach \i in {0,...,7}
	{
		\node[dot] at (\i,0)(\i){};
	}
	
\draw(4)  to [bend left=55]  node[midway,above]{1} (7);
\draw(0)  to [bend left=55] node[midway,above]{2}(2);
\draw(3)  to [bend left=55] node[midway,above]{3}(6);
\draw(1)  to [bend left=55] node[midway,above]{4}(5);
\draw(3.5,-1)node{$M_4=(3,4).M_3$};
\draw(3.5,-1.8)node{$lperm(M_4)=2431$};
\draw(3.5,-2.6)node{$ne(M_4)=0$};
\end{tikzpicture}
\caption{An example of the matching obtained by swapping left endpoints in a noncrossing matching, as in Definition \ref{def:swaps}. Note the $nep(M) = \{(1,2),(1,3),(1,4),(3,4)\}$, which defines the order in which left endpoints are swapped.}
\label{fig:swaps}
\end{figure}

\begin{lemma} \label{lem:adj}
Let $M$ be a noncrossing matching, let $M_0, M_1,\ldots,M_k$ be the sequence of matchings as defined in Definition~\ref{def:swaps}, and let $nep(M)=\{(a_{1},b_{1}),(a_{2},b_{2}),\ldots,(a_k,b_k)\}$. Then for every $i\in[k]$, $a_{i+1}, b_{i+1}$ appear in order and are adjacent in $lperm(M_i)$.
\end{lemma}

\begin{proof}
We use proof by induction. The result clearly holds for $M_0$. Assume that $a_{j+1},b_{j+1}$ are adjacent and in order in $lperm(M_j)$ for $0 \leq j <i$. First notice that this assumption immediately implies that $a_{i+1},b_{i+1}$ appear in order in $lperm(M_i)$.

To obtain a contradiction, assume that $a_{i+1}, b_{i+1}$ are not adjacent in $lperm(M_i)$. So, there exists some $c$ such that $c$ appears between $a_{i+1},b_{i+1}$. Notice that nestings in $M_0$ that are lexicographically smaller than $(a_{i+1},b_{i+1})$ have had their left endpoints swapped to obtain $M_i$. This fact will allow us to obtain a contradiction. We will consider three cases based on the size of the label $c$.

First assume that $a_{i+1}<b_{i+1}<c$. This implies that $c$ and $b_{i+1}$ appear out of order in $lperm(M_i)$, and must have been swapped. However, $(b_{i+1},c)$ is lexicographically larger than $(a_{i+1},b_{i+1})$, which gives a contradiction.

Next assume that $c<a_{i+1}<b_{i+1}$. In this case, $c$ and $a_{i+1}$ appear out of order in $lperm(M_i)$. The inductive assumption implies that at some step, $(c,a_{i+1})$ must have been swapped and were a nesting in $M_0$. However, this would also imply that $(c,b_{i+1})$ is a nesting in $M_0$ that is lexicographically smaller than $(a_{i+1},b_{i+1})$; applying this swap would reverse the order of $c,b_{i+1}$ as well. So, $c$ could not appear between $a_{i+1}$ and $b_{i+1}$ and we have a contradiction.

Finally assume that $a_{i+1}<c<b_{i+1}$. We need to consider two further subcases dependent on $rperm(M_0)$. If $c$ appears before $b_{i+1}$ in $rperm(M_0)$, then $c$ and $b_{i+1}$ are not nested edges by Lemma~\ref{lem:rperm}. However, $a_{i+1}$ nests both $c$ and $b_{i+1}$ in this case. It follows that $(a_{i+1},c)$ is a lexicographically smaller nesting in $nep(M_0)$ and $c,a_{i+1}$ must have already been swapped. As a result, they cannot appear in order in $M_i$, giving a contradiction.

If instead $b_{i+1}$ appears before $c$ in $rperm(M_0)$, the condition on noncrossing edges in $M_0$ and the fact that $a_{i+1}$ and $b_{i+1}$ are nested implies that $a_{i+1}$ and $c$ are also nested. Again, this is a lexicographically smaller nesting, which implies that $a_{i+1}$ and $c$ must appear out of order in $M_i$, giving a contradiction.
\end{proof}

\begin{lemma} \label{lem:nestings}
Let $M$ be a noncrossing matching, let $M_0, M_1,\ldots,M_k$ be the sequence of matchings as defined in Definition~\ref{def:swaps}, and let $nep(M)=\{(a_{1},b_{1}),(a_{2},b_{2}),\ldots,(a_k,b_k)\}$. Then  $nep(M_i)=\{(a_{i+1},b_{i+1}),\newline(a_{i+2},b_{i+2}),\ldots,(a_k,b_k)\}$; in particular, $ne(M_i)=k-i$.
\end{lemma}

\begin{proof}
We use proof by induction on $i$. It is straightforward to show that $nep(M_1)=\{(a_{2},b_{2}),\ldots \\,(a_k,b_k)\}$.

Now assume that $nep(M_{i-1})=\{(a_{i},b_{i}),\ldots,(a_k,b_k)\}$. By definition, we know that $(a_i,b_i)$ is no longer a nesting in $M_i=(a_i,b_i).M_{i-1}$. Additionally, since $(a_i,b_i)$ were adjacent in $lperm(M_{i-1})$ by Lemma~\ref{lem:adj}, we know that $(a_{i+1},b_{i+1}),(a_{i+2},b_{i+2}),\ldots,(a_k,b_k)$ are all still nestings in $M_i$ and that no additional nestings are created by swapping $(a_i,b_i)$. Therefore $nep(M_i)=\{(a_{i+1},b_{i+1}),\\ (a_{i+2},b_{i+2}),\ldots,(a_k,b_k)\}$.
\end{proof}

It follows that the matchings we generate by the swaps in Definition~\ref{def:swaps} form a set of representatives of the $k+1$ distinct nesting-similarity classes for matchings with the same LR-sequence as $M$. Let $\mathcal{NS}_n$ be the set of representatives of the nesting-similarity classes; that is $\mathcal{NS}_n = \{N \in \MM(n) \mid \text{$N=M_i$ for some noncrossing matching $M$ with at least $i$ nestings}\}$.  We can now define the other half of our bijection.

\begin{definition}
Define $\tau: \mathcal{NCN}_n \rightarrow \mathcal{NS}_n$ where, if $M$ is a noncrossing matching with $nep(M)=\{(a_{1},b_{1}),(a_{2},b_{2}),\ldots,(a_k,b_k)\}$, then $\tau(M,a_i,b_i)=M_i$ and $\tau(M,0,0)=M$.
\end{definition}

So, $\tau$ will take the noncrossing matching $M$, which has the maximum number of nestings of any matching with the same LR-sequence, and perform a sequence of left vertex swaps. Each swap converts precisely one nesting pair of edges into a crossing pair of edges; these swaps continue until the associated pair of edges is no longer nested. For an example of the mapping $\tau$, see Figure~\ref{fig:tau}.

\begin{theorem} \label{tau}
The mapping $\tau: \mathcal{NCN}_n \rightarrow \mathcal{NS}_n$ is a bijection. Additionally, if $M \in \mathcal{M}(n)$ is noncrossing, $\tau(M,0,0)=M$.
\end{theorem}

\begin{proof}
Given some $N \in \mathcal{NS}_n$, consider $M = nc(N)$. If $N=M$ (implying that $N$ is noncrossing), then $\tau^{-1}(N)=(N,0,0)$. Otherwise, by the definition of $\mathcal{NS}_n$, there exists some $i$ such that $N = M_i$. If $nep(M) =\{(a_{1},b_{1}),(a_{2},b_{2}),\ldots,(a_k,b_k)\}$, set $\tau^{-1}(N)=(M,a_i,b_i)$. 
\end{proof}

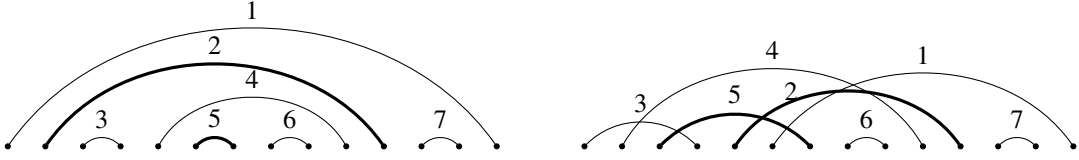
\begin{figure}[ht]
\centering

\begin{tikzpicture}[scale=0.5]
\foreach \i in {0,...,13}
	{
		\node[dot] at (\i,0)(\i){};
	}
	
\draw(0)  to [bend left=55]  node[midway,above]{1} (13);
\draw(1)[very thick]  to [bend left=55] node[midway,above]{2}(10);
\draw(2)  to [bend left=45] node[midway,above]{3}(3);
\draw(4)  to [bend left=60] node[midway,above]{4}(9);
\draw(5)[very thick]  to [bend left=45] node[midway,above]{5}(6);
\draw(7)  to [bend left=45] node[midway,above]{6}(8);
\draw(11)  to [bend left=45] node[midway,above]{7}(12);
\end{tikzpicture}\hspace{1cm}
\begin{tikzpicture}[scale=0.5]
\foreach \i in {0,...,13}
	{
		\node[dot] at (\i,0)(\i){};
	}
	
\draw(5)  to [bend left=55]  node[midway,above]{1} (13);
\draw(4)[very thick]  to [bend left=55] node[midway,left=15pt]{2}(10);
\draw(0)  to [bend left=45] node[midway,above]{3}(3);
\draw(1)  to [bend left=60] node[midway,above]{4}(9);
\draw(2)[very thick]  to [bend left=45] node[midway,above]{5}(6);
\draw(7)  to [bend left=45] node[midway,above]{6}(8);
\draw(11)  to [bend left=45] node[midway,above]{7}(12);
\end{tikzpicture}
\caption{An example of the bijection $\tau$, mapping from the noncrossing matching on the left, with chosen nesting pair $(2,5)$ to the representative of the nesting-similarity class with 5 nestings and LR-sequence LLLRLLRLRRRLRR on the right.}
\label{fig:tau}
\end{figure}

Now, combining Theorems \ref{phi} and \ref{tau}, we immediately obtain our desired result.

\begin{theorem}
The map $\sigma = \tau \circ \phi:\mathcal{LP}_n \rightarrow \mathcal{NS}_n$ is a bijection between L \& P matchings and nesting-similarity classes.
\end{theorem}

An example of this composition can be seen in Figure~\ref{fig:WholeBijection}.

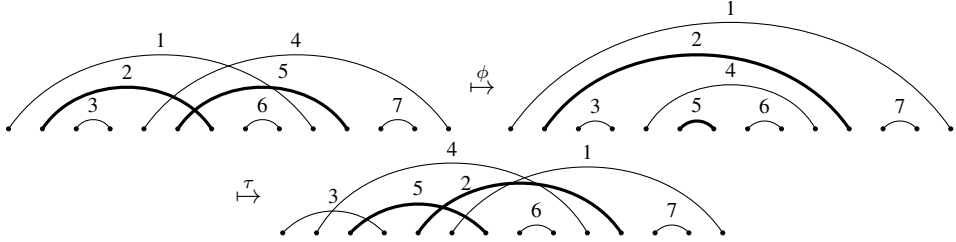
\begin{figure}[ht]
\centering
{\footnotesize
\begin{tikzpicture}[scale=0.45]
\foreach \i in {0,...,13}
	{
		\node[dot] at (\i,0)(\i){};
	}
	
\draw(0)  to [bend left=55]  node[midway,above]{1} (9);
\draw(1)[very thick]  to [bend left=55] node[midway,above]{2}(6);
\draw(2)  to [bend left=55] node[midway,above]{3}(3);
\draw(4)  to [bend left=55] node[midway,above]{4}(13);
\draw(5)[very thick]  to [bend left=55] node[midway,right=8pt, above]{5}(10);
\draw(7)  to [bend left=55] node[midway,above]{6}(8);
\draw(11)  to [bend left=55] node[midway,above]{7}(12);

\draw(14,1.5) node{{\normalsize$\overset{\phi}{\mapsto}$}};
\end{tikzpicture}
\begin{tikzpicture}[scale=0.45]
\foreach \i in {0,...,13}
	{
		\node[dot] at (\i,0)(\i){};
	}
	
\draw(0)  to [bend left=55]  node[midway,above]{1} (13);
\draw(1)[very thick]  to [bend left=55] node[midway,above]{2}(10);
\draw(2)  to [bend left=45] node[midway,above]{3}(3);
\draw(4)  to [bend left=60] node[midway,above]{4}(9);
\draw(5)[very thick]  to [bend left=45] node[midway,above]{5}(6);
\draw(7)  to [bend left=45] node[midway,above]{6}(8);
\draw(11)  to [bend left=45] node[midway,above]{7}(12);
\end{tikzpicture}
\begin{tikzpicture}[scale=0.45]
\foreach \i in {0,...,13}
	{
		\node[dot] at (\i,0)(\i){};
	}
	
\draw(5)  to [bend left=55]  node[midway,above]{1} (13);
\draw(4)[very thick]  to [bend left=55] node[midway,left=15pt]{2}(10);
\draw(0)  to [bend left=45] node[midway,above]{3}(3);
\draw(1)  to [bend left=60] node[midway,above]{4}(9);
\draw(2)[very thick]  to [bend left=45] node[midway,above]{5}(6);
\draw(7)  to [bend left=45] node[midway,above]{6}(8);
\draw(11)  to [bend left=45] node[midway,above]{7}(12);
\draw(-1,1.25) node{{\normalsize$\overset{\tau}{\mapsto}$}};
\end{tikzpicture}
}
\caption{Above is the result of first applying $\phi$ to an L \& P matching to obtain a noncrossing matching with an indicated nesting pair. Then, we see the result of applying $\tau$ to the noncrossing matching with indicated nesting pair to obtain a representative of a nesting equivalence class. Composed, this is the mapping $\sigma$.}
\label{fig:WholeBijection}
\end{figure}

 The images under the map $\sigma$ also form a set of representatives for the nesting similarity classes $\mathcal{NS}_n$. Although Klazar proved that the classes exist, representatives of those classes were not explicitly provided, and now we have done so.

\begin{corollary}
Since the bijection $\sigma$ has an intermediate step at the noncrossing matching associated to a matching M, $\sigma$ has the following properties:

\begin{itemize}
\item{If $M$ is a noncrossing matching, then $\sigma(M) = M$.}
\item{The LR-sequence of M is the LR-sequence of $\sigma(M)$.}
\end{itemize}

\end{corollary}

Other properties that are sometimes preserved in bijections between matchings, such as number of nestings or number of crossings, are not preserved by $\sigma$. However these statistics all fail to be equidistributed between the two sets $LP_n$ and $NS_n$, so no other bijection exists that preserves them.

In the larger context, we note that none of the other Largest Hairpin Family matchings (LHF, D\&P, R\&G, C\&C) have closed forms for their enumeration sequences. However given that our map $\sigma$ relates an L \& P matching to a noncrossing matching and a nesting edge pair, and the fact that all of these families are constructed inductively, we believe it may be possible to find a similar mapping involving a noncrossing matching and some other matching property. The authors hope to explore this possibility in future work.

The authors would like to thank the organizers of the Dagstuhl Seminar $16071$: Pattern Avoidance and Genome Sorting, which is where the problem was first presented, and where the authors had generous working time to explore the ideas in this paper. The authors also thank the referees for their helpful comments which improved the readability of the paper. 

\bibliographystyle{plain}
\bibliography{RNASS}

\end{document}